\documentclass[10pt, fleq]{amsart}
\usepackage{amssymb,amsmath,latexsym,amsbsy,color,enumerate}
\numberwithin{equation}{section}

\def\CC{\mbox{$C\hspace{-.11in}\protect\raisebox{.5ex}{\tiny$/$}
\hspace{.06in}$}}

\def\h{\hspace*{.24in}}

\newtheorem{theorem}{Theorem}

\begin{document}
 \title{Note on potential theory for functions in Hardy classes}
\author{Tuyen Trung Truong}
    \address{Department of Mathematics, Indiana University Bloomington, IN 47405 USA}
 \email{truongt@indiana.edu}
\thanks{}
    \date{\today}
    \begin{abstract}
The purpose of this note is to show that the set functions defined in \cite{trong-tuyen} can be suitably extended to all subsets $E$ of the unit disk
$\mathbb{D}$. In particular we obtain uniform nearly-optimal estimates for the following quantity
\begin{eqnarray*}
D_p(E,\varepsilon, R) = \sup \{\sup_{|z| \leq R}|g(z)|: g\in H^p,~||g||_{H^p}\leq 1,~(1-|\zeta |)|g(\zeta)| \leq \varepsilon\;\forall \zeta\in E\}.
\end{eqnarray*}
\end{abstract}
\maketitle
\section{Introduction and results.}       
The (weighted) potential theory in $\mathbb{C}$ (see for example \cite{ransford} and \cite{Tsuji}) is a powerful tool used to obtain estimates about the
growth of subharmonic functions. This tool can also be applied to estimate the growth of analytic functions with great success. However in some problems
related to special classes of analytic functions (e.g. Hardy classes, Bergman classes) it seems that we need other tools to obtain better bounds.

In \cite{trong-tuyen}, when considering the stability problem for functions in Hardy classes, it is shown that other set functions defined there are more
fitter in order to obtain good estimates. Let us briefly review the notations and results in \cite{trong-tuyen}.

Let $E$ be a non-Blaschke subset of the unit disc $\mathbb{D}$ of the complex plane $\CC$, that is $E$ contains a sequence $(z_j)$ satisfying the
condition
\[
\displaystyle\sum_{j=1}^\infty(1-|z_j|) = \infty.
\]

Fixed $1\leq p\leq \infty$, recall that the Hardy space $H^p(\mathbb{D})$ is the space of all holomorphic functions $g$ on $\mathbb{D}$ for which
$\|g\|_p$ $<$ $\infty$, where
\begin{align*}
\|g\|_p &= \lim_{r\uparrow 1} \left\{\frac{1}{2\pi}\int_0^{2\pi}|g(re^{i\theta})|^p d\theta\right\}^{1/p}\h(1 \leq p < \infty),\\
\|g\|_{\infty} &= \lim_{r\uparrow 1} \sup_{\theta} |g(re^{i\theta})|.
\end{align*}
For convenience, from now on, we will denote $H^p(\mathbb{D})$ by $H^p$. We define $\mathcal{A}^p$ to be the functions in $H^p$ with norm $1$, that is
\begin{equation}
\mathcal{A}^p=\{f:~f\in H^p, ||f||_{H^p}=1\}.\label{EquationSpaceAp}
\end{equation}
If $f\in \mathcal{A}^p$ it follows that
\begin{equation}
|f(z)|\leq \frac{1}{(1-|z|^2)^{1/p}}, \label{EquationEstimateForFunctionsInSpaceAp}
\end{equation}
for all $z\in \mathbb{D}$.

In \cite{trong-tuyen}, the following problem is considered: Find good (both upper and lower) bounds for the quantity
\begin{equation}
C_p(E,\varepsilon, R) = \sup \{\sup_{|z| \leq R}|g(z)|: g\in \mathcal{A}^p,|g(\zeta)| \leq \varepsilon\;\forall \zeta\in E\}, \label{CeRdefinition}
\end{equation}
for positive $\varepsilon$ and $R$ in $(0, 1)$.

Let $E_0$ be the set of non-tangential limit points of $E$, that is  points $\zeta$ of $\partial \mathbb{D}$ being such that there exists a sequence
$(z_n)$ in $E$ which tends nontangentially to $\zeta$, that is, such that
\begin{eqnarray*}
z_n\rightarrow \zeta,~|z_n-\zeta|=O(1-|z_n|).
\end{eqnarray*}

Let $m(E_0)$ be the Lebesgue measure of $E_0$ as a subset of $\mathbb{D}$. If $m(E_0)>0$ then $C_p(E,\epsilon ,R)$ can be estimated using harmonic
measures of $E_0$ (see Appendix in \cite{trong-tuyen}). In order to deal with the case $m(E_0)=0$, some set functions $M_n(E)$ were introduced in Section
3 in \cite{trong-tuyen}, using weighted finite Blaschke products
\begin{eqnarray*}
B_q(Z_n,z)=q(z)\prod _{j=1}^n\frac{z-z_j}{1-\overline{z_j}z},
\end{eqnarray*}
where $Z_n=\{z_1,\ldots ,z_n\}$, and $q(z)$ is a function provided by a result of Hayman\cite{hay} satisfying the following two conditions:

i) $q(z)$ is analytic in $\mathbb{D}$

ii) $q(z)$ is bounded in $\mathbb{D}$ with sup-norm $1$, $q(z)\not= 0$ for all $z\in \mathbb{D}$, and
\begin{eqnarray*}
\lim _{z\in E,|z|\rightarrow 1}q(z)=0.
\end{eqnarray*}

The set functions $M_n(E)$ are analogous to the set functions defined in (weighted) potential theory in $\mathbb{D}$. If $E$ is a compact subset in
$\mathbb{D}$ then $M_n(E)$ can be defined using the potential theory in the unit disk as discussed in \cite{Tsuji}. However when $E$ is not relatively
compact in $\mathbb{D}$, it turns out that $M_n(E)$ are different from their counterparts in weighted potential theory for $\mathbb{D}$ (see Section 3 in
\cite{trong-tuyen} for more detailed).

The functions $M_n(E)$ as defined in \cite{trong-tuyen} requiring $q(z)$ to satisfy both conditions i) and ii) can not be extended to sets $E$ with
$m(E_0)>0$. That is because when $m(E_0)>0$, by the above mentioned result of Hayman, there is no function $q(z)$ satisfying both i) and ii). However if
we use any function $q(z)$ satisfying condition ii) but not necessarily analytic, we can still define set functions $M_n(E)$ as in Section 3 in
\cite{trong-tuyen}.

The most natural function $q(z)$ satisfying condition ii) is the function $q(z)=1-|z|$. Note that $q(z)$ is not analytic, however it is independent of
sets $E$. Hence we can define set functions $\widetilde{M}_n(E):=M_n(E)$ (we use the notation $\widetilde{M}$ to signify that we are using the special
function $q(z)=1-|z|$ in defining $M_n(E)$) for any subset $E$ of $\mathbb{D}$. It turns out that these sets functions allow us to obtain uniform
nearly-optimal estimates for a quantity similar to $C_p(E,\epsilon ,R)$. Let us state the result more explicitly in the below.

Fix from now on $q(z)=1-|z|$. Let $E$ be any non-Blaschke subset of the unit disc $\mathbb{D}$ of the complex plane $\CC$. Consider the following
quantity
\begin{equation}
D_p(E,\varepsilon, R) = \sup \{\sup_{|z| \leq R}|g(z)|: g\in \mathcal{A}^p,~(1-|\zeta |)|g(\zeta)| \leq \varepsilon\;\forall \zeta\in E\},
\label{DeRdefinition}
\end{equation}
for positive $\varepsilon$ and $R$ in $(0, 1)$. In view of (\ref{EquationEstimateForFunctionsInSpaceAp}), the quantity $D_p(E,\varepsilon, R)$ is natural
to be considered.

Define \begin{equation} g(E,\epsilon ,R,q)=\sup \{\sup _{|z|\leq R}|B(Z_n;z)|:~n\in \mathbb{N},~Z_n\in E^n,~|B_q(Z_n;\zeta )|\leq \epsilon~\forall \zeta
\in E\}.\label{EquationOfFunctionG}
\end{equation}
\begin{theorem}
There exists an absolute constant $\sigma >0$, and there exist constants $K,\alpha ,\epsilon _0>0$ depending only on $p$ and $R$ such that for all
$0<\epsilon <\epsilon _0$ and any non-Blaschke subset $E$ of $\mathbb{D}$
\begin{eqnarray*}
g(E,\epsilon ,R)\leq D_p(E,\epsilon ,R)\leq K \times g^{\alpha}(E,\epsilon ^{\sigma },R).
\end{eqnarray*}
\label{TheoremMain}\end{theorem}
\begin{proof}
By proof of Part 2 of Proposition 4 in \cite{trong-tuyen}, there exists absolute constants $C,\sigma >0$ such that
\begin{eqnarray*}
\widetilde{M}_n(E)\leq Cn^{-\sigma }
\end{eqnarray*}
for any $n$ big enough. The inequality holds uniformly for all non-Blaschke sets $E$.

Now exploring the functions $c_{p,k}(Z_n,z)$ in formula (2.2) in \cite{trong-tuyen} shows that estimates (4.1) and (4.2) in \cite{trong-tuyen} are still
true with $C_p(E,\epsilon ,R)$ replaced by $D_p(E,\epsilon ,R)$. Then applying the proof of Theorem 4 in \cite{trong-tuyen} completes the proof of
Theorem \ref{TheoremMain}.
\end{proof}
{\bf Acknowledgement} The author would like to thank Professor Norman Levenberg and Professor Yuril Lyubarskii for their helpful comments.

\end{document}